\documentclass[11pt]{amsart}
\usepackage[left=3.6cm,top=3cm,right=3.6cm,bottom=3cm,bindingoffset=0.5cm]{geometry}
\usepackage[alphabetic,lite]{amsrefs}
\usepackage{amsmath}
\usepackage{amsthm}
\usepackage{mathrsfs}
\usepackage{amssymb}

\usepackage{xcolor}

\newcommand{\bC}{\mathbb{C}}
\newcommand{\bP}{\mathbb{P}}
\newcommand{\pu}{\bP^1_{\bC}}

\newcommand{\ppu}{\pu \times \pu}
\newcommand{\ccs}{(\bC^*)^2}

\newcommand{\xg}{x_g}
\newcommand{\yg}{y_g}
\newcommand{\vg}{\xg,\yg}

\newcommand{\xl}{x_l}
\newcommand{\yl}{y_l}
\newcommand{\hop}{HOMFLY polynomial}
\newcommand{\hos}[1]{\left( #1 \right)}
\newcommand{\mir}[1]{\overline{#1}}

\newcommand{\duek}{3_1}
\newcommand{\due}{\hos{3_1}}
\newcommand{\duep}{\hos{\mir{\duek}}}

\newcommand{\quattro}{\hos{4_1}}
\newcommand{\cinque}{\hos{5_2}}
\newcommand{\sei}{\hos{7_5}}
\newcommand{\sette}{\hos{8_{16}}}
\newcommand{\otto}{\hos{10_{150}}}
\newcommand{\nove}{\hos{8_{17}}}

\newcommand{\tredici}{\hos{8_{14}}}

\newcommand{\sedici}{\hos{9_{33}}}
\newcommand{\diciassette}{\hos{8_3}}

\newcommand{\dodicip}{\hos{\mir{7_5}}}
\newcommand{\quattordicip}{\hos{\mir{7_7}}}
\newcommand{\duecentoventip}{\hos{\mir{10_{136}}}}
\newcommand{\seicentotrentaseip}{\hos{\mir{11n_{20}}}}
\newcommand{\settecentoquarantap}{\hos{\mir{11n_{124}}}}

\newcommand{\qs}{\hos{9_{12}}}
\newcommand{\qvq}{\hos{11a_{175}}}
\newcommand{\qvc}{\hos{11a_{176}}}
\newcommand{\qsn}{\hos{11a_{220}}}
\newcommand{\ccc}{\hos{11a_{306}}}
\newcommand{\ncd}{\hos{12a_{151}}}
\newcommand{\nss}{\hos{12a_{165}}}
\newcommand{\ms}{\hos{12a_{259}}}
\newcommand{\mcu}{\hos{12a_{300}}}
\newcommand{\mdsd}{\hos{12a_{471}}}
\newcommand{\mtsei}{\hos{12a_{505}}}
\newcommand{\mtsette}{\hos{12a_{506}}}
\newcommand{\mtsed}{\hos{12a_{515}}}
\newcommand{\mtd}{\hos{12a_{517}}}
\newcommand{\mtts}{\hos{12a_{535}}}
\newcommand{\dccu}{\hos{12n_{462}}}
\newcommand{\dcon}{\hos{12n_{500}}}

\usepackage{hyperref}
\hypersetup{
    colorlinks=true,       
    linkcolor=blue,          
    citecolor=magenta,        
    filecolor=green,      
    urlcolor=purple           
}

    \newtheorem{Lem}{Lemma}
    \newtheorem{Prop}[Lem]{Proposition}
    \newtheorem{Thm}[Lem]{Theorem}

\theoremstyle{definition}

    \newtheorem{Rem}[Lem]{Remark}

\address{Douglas Blackwell}
\email{dougpblackwell@gmail.com}

\address{Damiano Testa, Mathematics Institute, University of Warwick, Coventry, CV4 7AL, UK}
\email{adomani@gmail.com, d.testa@warwick.ac.uk}

\begin{document}

\title[]{Factors of HOMFLY polynomials}

\author[]{Douglas Blackwell}
\author[]{Damiano Testa}

\begin{abstract}
We study factorizations of \hop{}s of certain knots and oriented links.  We begin with a computer analysis of knots with at most~$12$ crossings, finding 17 non-trivial factorizations.  Next, we give an irreducibility criterion for \hop{}s of oriented links associated to 2-connected plane graphs.
\end{abstract}



\maketitle

\section*{Introduction}

Several properties of knots and links are encoded using polynomial invariants.  Many of the properties of these polynomials are of a combinatorial nature, such as the degree, or coordinate dependent, such as special evaluations.  For a few examples, see the Morton-Franks-Williams inequality \cites{frwi,mor,morp}, the slope conjecture~\cite{gar}, some evaluations of link polynomials~\cite{limi}, degree computations~\cite{vdv}.

In this paper, we propose to study a geometric property: irreducibility of the \hop.  Thus, we view the \hop{} of an oriented link as a plane algebraic curve and we ask if the curve is irreducible.  Since the \hop{} is really a Laurent polynomial, we disregard the coordinate axes in our analysis.

First, we perform a computer analysis of \hop{}s of the 2977 knots with at most 12 crossings: we find 17 non-trivial factorizations (Table~\ref{tab:fatto}).  To obtain the polynomials, we consulted the databases KnotInfo~\cite{kni} and KnotAtlas~\cites{kat}.  To factor them, we used the computer algebra program {\sc{Magma}}~\cite{magma}.

Second, we give a sufficient criterion for irreducibility of the \hop{}s of oriented links associated to plane graphs by Jaeger in~\cites{jaeger}.  A standard construction of Jaeger (\cite{jaeger}*{page 649}) associates to each connected plane graph~$G$ an oriented link diagram $D(G)$.  Jaeger shows that the \hop{} $P(D(G),x,y,z)$ can be computed from the Tutte polynomial $T_G(x,y)$ of~$G$ using the formula
\begin{equation} \label{e:sost}
P(D(G),x,y,z) =
\left( \frac{y}{z} \right)^{|V(G)|-1}
\left( -\frac{z}{x} \right)^{|E(G)|}
T_G \left(
-\frac{x}{y}
,
1-\frac{(x+y)y}{z^2}
\right).
\end{equation}
Thus, ignoring powers of $x,y,z$, the irreducibility of $T_G(x,y)$ is a necessary condition for the irreducibility of $P(D(G),x,y,z)$.  The Tutte polynomial of a 2-connected graph is irreducible by a result of Merino, de Mier and Noy (\cite{MMN}*{Theorem~1}).  Hence, we reduce the study of the irreducibility of the \hop{} $P(D(G),x,y,z)$ to understanding how the substitution in~\eqref{e:sost} interacts with the Tutte polynomial.

To achieve our goal, Proposition~\ref{p:simpl} simplifies formula~\eqref{e:sost}.  The verification of the identity is entirely mechanical, but our arguments hinge on the {\emph{existence}} of such a simple final result.  Next, Lemma~\ref{l:riqua} gives a sufficient criterion for irreducibility of polynomials, adapted to our needs.  We combine these statements in Theorem~\ref{t:irri}, the main criterion for irreducibility of \hop{}s of this paper.

\subsection*{Notation}

Let~$K$ be a knot.  To simplify our formulas, we denote by $\hos{K}$ the \hop{} of the knot~$K$, defined using the convention of~\cite{homfly}*{\sc{Main Theorem}}.  Thus, $\hos{K}$ is a homogeneous Laurent polynomial of degree~$0$ in $\mathbb{Z}[x^{\pm 1},y^{\pm 1},z^{\pm 1}]$: the numerator of~$\hos{K}$ is a homogeneous polynomial in $x,y,z$ and the denominator of~$\hos{K}$ is a monomial of the same degree as the numerator.  We denote by $\mir{K}$ the mirror image of the knot.  Recall that the \hop{} of the mirror image of a knot~$K$ satisfies the identity
\[
\hos{\mir{K}} (x,y,z) = \hos{K} (y,x,z).
\]
To identify knots, we follow the notation of \href{https://knotinfo.math.indiana.edu/}{KnotInfo}~\cite{kni}.  For convenience, we reproduce the part of the convention that is relevant for us:

\smallskip

\begin{center}
\begin{minipage}{350pt}
{\textsf{\small{``For knots with~$10$ or fewer crossings, we use the classical names, as tabulated for instance by Rolfsen, eliminating the duplicate $10_{162}$ from the count. For~$11$ crossing knots, we use the Dowker-Thistlethwaite name convention, based on the lexigraphical ordering of the minimal Dowker notation for each knot.''}}}
\end{minipage}
\end{center}

\smallskip

For instance, $4_1$ is the Figure-eight knot, while
\[
\duep = \frac{z^2}{y^2} - 2\frac{x}{y} -\frac{x^2}{y^2}
\]
is the \hop{} of the left-handed Trefoil knot.

{\emph{Caution.}}
The convention for the \hop{} used in~\cite{kni} differs from the one that we use.  We obtain the \hop{} $\hos{K}_{KI}$ of the knot~$K$ tabulated in~\cite{kni} by the substitution
\[
\hos{K}_{KI} = \hos{K} (v^{-1},v,-z).
\]
This happens in the background and plays almost no role in the arguments.

\section{Knots with up to 12 crossings}

We started this project wondering about irreducibility of \hop{}s.  A quick calculation with a computer, shows that the {\hop} of the knot $9_{12}$ is the product of the \hop{}s of the knots $4_1$ (Figure-eight knot) and $5_2$ (3-twist knot):
\[
\hos{9_{12}} = \hos{4_{1}} \hos{5_{2}}.
\]
Simlarly, also the identity
\[
\hos{11a_{175}} = \hos{3_{1}} \hos{8_{16}}
\]
holds.  Systematizing these results, we analyzed the knots with up to 12 crossings, using the database~\cite{kni}.  Out of these 2977 \hop{}s, 17 are reducible.  Each one of these 17 reducible polynomials is a product of previous members of the database.  When checking for divisibility, we work in the Laurent polynomial ring $\mathbb{Z}[x^{\pm 1}, y^{\pm 1},z^{\pm 1}]$, that is, we disregard powers, positive or negative, of $x,y,z$.  Still, the factorizations that we find are correct as stated: there is no need to adjust by multiplying by a unit.  We collect this data in Table~\ref{tab:fatto}.

\begin{table}[h]
\[
\begin{array}{|r@{\; = \;}l|}
\hline
{\vphantom{{W^W}^W}}
\qs & \quattro \cinque \\[5pt]
\qvq & \due \sette \\[5pt]
\qvc & \due \nove \\[5pt]
\qsn & \quattro \sei \\[5pt]
\ccc & \duep \sette \\[5pt]
\ncd & \cinque \quattordicip \\[5pt]
\hline
\end{array}
\quad
\begin{array}{|r@{\; = \;}l|}
\hline
{\vphantom{{W^W}^W}}
\nss & \cinque \duecentoventip \\[5pt]
\ms & \quattro \seicentotrentaseip \\[5pt]
\mcu & \quattro \tredici \\[5pt]
\mdsd & \quattro \diciassette \\[5pt]
\mtsei & \duep \sedici \\[5pt]
\mtsette & \quattro \nove \\[5pt]
\hline
\end{array}
\quad
\begin{array}{|r@{\; = \;}l|}
\hline
{\vphantom{{W^W}^W}}
\mtsed & \due \settecentoquarantap \\[5pt]
\mtd & \quattro \otto \\[5pt]
\mtts & \quattro \sette \\[5pt]
\dccu & \quattro^2 \\[5pt]
\dcon & \due \dodicip \\[20pt]
\hline
\end{array}
\]
\caption{Factorizations of \hop{}s}
\label{tab:fatto}
\end{table}

In particular, the \hop{} $\dccu$ is the only one having a repeated irreducible factor.  We observe also that the Kauffman polynomials of the 2977 knots with at most 12 crossings are all irreducible.

\section{Graphs and oriented link diagrams}

In this section, we prove a criterion for the irreducibility of the \hop{}s of certain oriented links associated to plane graphs.

To argue irreducibility, we exploit the morphism $\ccs \dashrightarrow \ccs$ appearing in \cite{jaeger}*{Proposition~1}:
\begin{eqnarray*}
J_0 \colon \quad
\ccs
& \longrightarrow &
\ccs \\{}
(x,y)
& \longmapsto &
\left(
-\frac{x}{y} , 1 - (x+y)y \right).
\end{eqnarray*}
We simplify the expression of~$J_0$ by changing coordinates on the domain and codomain of~$J_0$.  Denote by $\Xi \colon \ppu \dashrightarrow \ccs$ the birational map
\begin{eqnarray*}
\Xi \colon \quad
\ppu
& \dashrightarrow &
\ccs \\{}
\left(
[x_0,x_1],[y_0,y_1]
\right)
& \longmapsto &
\left(
\frac{y_0}{y_1} - \frac{x_1 y_0}{x_0 y_1} ,
\frac{x_1 y_0}{x_0 y_1}
\right)
\end{eqnarray*}
with birational inverse
\begin{eqnarray*}
\Xi^{-1} \colon \quad
\ccs
& \longrightarrow &
\ppu \\{}
(\xl,\yl)
& \longmapsto &
\left(
[\xl + \yl , \yl],[\xl + \yl , 1]
\right).
\end{eqnarray*}
Denote by $\Sigma \colon \ccs \to \ppu$ the birational morphism
\begin{eqnarray*}
\Sigma \colon \quad
\ccs
& \longrightarrow &
\ppu \\{}
(\vg)
& \longmapsto &
\bigl(
[1-\xg,1],[(1-\xg)(1-\yg),1]
\bigr)
\end{eqnarray*}
with birational inverse
\begin{eqnarray*}
\Sigma^{-1} \colon \quad
\ppu & \dashrightarrow &
\ccs
\\{}
\bigl(
[x_0,x_1],[y_0,y_1]
\bigr)
& \longmapsto &
\left(
1-\frac{x_0}{x_1} ,
1- \frac{x_1y_0}{x_0y_1}
\right).
\end{eqnarray*}

Define the morphism $J \colon \ppu \to \ppu$ by setting
\begin{eqnarray*}
J \colon \quad
\ppu
& \longrightarrow &
\ppu \\{}
\left(
[x_0,x_1] , [y_0,y_1]
\right)
& \longmapsto &
\left(
\left[ x_0, x_1 \right] ,
\left[ y_0^2 , y_1^2 \right]
\right)
.
\end{eqnarray*}

\begin{Prop} \label{p:simpl}
The rational maps $J$ and $\Sigma \circ J_0 \circ \Xi$ coincide.
\end{Prop}

\begin{proof}
This is a matter of a simple substitution, using the definition of the involved maps.
\end{proof}

The morphism~$J$ is finite of degree~$2$ and it is branched over the divisor $R \subset \ppu$ with equation $y_0y_1 = 0$.

In our argument for irreducibility, we exploit the following easy algebraic lemma.

\begin{Lem} \label{l:riqua}
Let $C \subset \ppu$ be an irreducible curve, defined by the equation $F(x_0,x_1,y_0,y_1)=0$.  Assume that the polynomial~$F$ is bihomogeneous of degree~$a$ in $x_0,x_1$ and of degree~$b$ in $y_0,y_1$.  If the curve with equation $F(x_0,x_1,y_0^2,y_1^2)=0$ is reducible, then the two polynomials $F(x_0,x_1,1,0)$ and $F(x_0,x_1,0,1)$ are squares.  In particular, the degree~$a$ is even.
\end{Lem}

\begin{proof}
We cover $\ppu$ by~$4$ standard affine charts isomorphic to~$\mathbb{A}^2_{\bC}$, by setting one among~$x_0$ or~$x_1$ to~$1$ and also one among~$y_0$ or~$y_1$ to~$1$.  Fix one of these charts.  The bihomogeneous polynomial~$F$ becomes an irreducible polynomial $f(x,y) \in \bC[x,y]$.  To prove the result, we show that if the polynomial $f(x,y^2)$ is reducible, then $f(x,0) \in \bC[x]$ is a square.

Let $g(x,y) \in \bC[x,y]$ be an irreducible factor of $f(x,y^2)$.  Since $f(x,y^2)$ is not irreducible, we deduce the inequality $g(x,y) \neq f(x,y^2)$.  Separating the terms of $g(x,y)$ with an even and an odd exponent of~$y$, we find polynomials $g_0(x,y^2)$ and $g_1(x,y^2)$ such that the identity
\[
g(x,y) = g_0(x,y^2) + y g_1(x,y^2)
\]
holds.  If $g_0(x,y^2)$ vanishes, then we are done.  Suppose therefore that $g_0(x,y^2)$ is not the zero polynomial.  If $g_1(x,y^2)$ vanishes, then $g_0(x,y^2)$ is a proper factor of $f(x,y^2)$; as a consequence, $g_0(x,y)$ is a proper factor of $f(x,y)$, contradicting the irreducibility of $f(x,y)$.  It follows that $g(x,-y)$ is a polynomial that is not proportional to $g(x,y)$ and that also divides $f(x,y^2)$.  By irreducibility of $g(x,y)$, we deduce that the product $g(x,y)g(x,-y)$ divides $f(x,y^2)$.  By irreducibility of $f(x,y)$, we deduce that the product $g(x,y)g(x,-y)$ actually equals $f(x,y^2)$.
We therefore find
\[
f(x,y^2) = g_0(x,y^2)^2 -y^2 g_1(x,y^2).
\]
Setting~$y$ to~$0$, we conclude that the identity $f(x,0) = g_0(x,0)^2$ holds, as needed.
\end{proof}

\begin{Rem}
The statement above still holds replacing the complex numbers by any field~$k$.  If the characteristic of~$k$ is different from~$2$, then the given proof goes through essentially unchanged.  If the characteristic of~$k$ is~$2$, then the statement follows from \cite{stacks}*{\href{https://stacks.math.columbia.edu/tag/0BRA}{Tag 0BRA}}: in this case, the morphism~$J$ is purely inseparable and hence a homeomorphism.
\end{Rem}

Suppose that~$G$ is a connected plane graph and denote by $T_G(x,y)$ the Tutte polynomial of~$G$.  Denote by $D(G)$ the associated link diagram constructed by Jaeger~\cite{jaeger}.  We do not reproduce here the construction of~$D(G)$: we refer the interested reader to \cite{jaeger}*{Section~2}.  All that we need is that the identity
\[
\hos{D(G)} = 
\left( \frac{y}{z} \right)^{|V(G)|-1}
\left( -\frac{z}{x} \right)^{|E(G)|}
T_G \left(
-\frac{x}{y} , 1 - \frac{(x+y)y}{z^2}
\right)
\]
holds (see \cite{jaeger}*{Proposition~1}).

We are interested in the irreducibility of the HOMFLY polynomial of the link diagram $D(G)$.

We view HOMFLY polynomials as elements of the Laurent polynomial ring $L = \bC[x^{\pm 1},y^{\pm 1},z^{\pm 1}]$.  Thus, irreducibility of a non-zero element $f \in L$ means that any factorisation $f=gh$, with $g,h \in L$, implies that either~$g$ or~$h$ has the form $\alpha \, x^ay^bz^c$, with $\alpha \in \bC$ and $a,b,c \in \mathbb{Z}$.

For a polynomial $t(\vg)$ in the coordinates $\vg$ of $\ccs$, we want to read the information about the ramification of the morphism~$J$.  Thus, we take the strict transform under~$\Sigma^{-1}$ of the vanishing set of $t(\vg)$ and intersect the closure of this locus with $y_0=0$ and $y_1=0$.  We summarize the outcome of this easy computation in the following lemma for future reference.

\begin{Lem} \label{l:res}
Let $t(\vg) = \sum _{i,j} t_{ij} \xg^i \yg^j$ be a polynomial in $\bC[\vg]$ and let $T \subset \ccs$ be the curve defined by the equation $t(\vg) = 0$.  Let $d \in \mathbb{N}$ be the largest exponent of~$\yg$ among the monomials appearing in $t(\vg)$ with non-zero coefficient. 
\begin{itemize}
\item
An equation for the intersection $\overline{\Sigma (T)} \cap \{y_0=0 , \; x_0 x_1 \neq 0 \}$ is
\[
t \left( 1-\frac{x_0}{x_1},1 \right) = 0.
\]
\item
An equation for the intersection $\overline{\Sigma (T)} \cap \{y_1=0 , \; x_0 x_1 \neq 0 \}$ is
\[
\sum_{i}t_{i d} \left( 1-\frac{x_0}{x_1} \right)^i =0.
\]
\end{itemize}
\end{Lem}

\begin{proof}
We obtain an equation vanishing of the curve $\overline{\Sigma(T)}$ by setting to~$0$ the numerator of the evaluation $t \left(
1-\frac{x_0}{x_1} ,
1- \frac{x_1y_0}{x_0y_1}
\right)$.  It is now a matter of a straightforward computation to check that the stated identities hold.
\end{proof}

Let~$G$ be a connected plane graph.  We define two curves $P_G \subset \ccs$ and $\mathscr{P}_G \subset \ppu$.  We set
\[
P_G \colon \quad
\hos{D(G)} (\xl, \yl,1) = 0
\quad \subset \quad
\ccs ,
\]
and
\[
\mathscr{P}_G =
\overline{\Xi^{-1} (P_G)}
\quad \subset \quad
\ppu .
\]
Similarly, we define two curves $T_G \subset \ccs$ and $\mathscr{T}_G \subset \ppu$.  We set
\[
T_G \colon \quad
T_G \left( \vg \right) = 0
\quad \subset \quad
\ccs ,
\]
and
\[
\mathscr{T}_G =
\overline{\Sigma (T_G)}
\quad \subset \quad
\ppu .
\]
Thus,~$\mathscr{P}_G$ and~$\mathscr{T}_G$ are, essentially, the vanishing of the HOMFLY polynomial of the oriented link $D(G)$ and of the Tutte polynomial of~$G$, respectively.

As a consequence of the definitions and of \cite{jaeger}*{Proposition~1}, we deduce that there is a diagram
\[
\begin{array}{ccc}
\ppu
& \stackrel{J}{\longrightarrow} &
\ppu \\{}
\cup
& &
\cup \\
\mathscr{P}_G & \longrightarrow & \mathscr{T}_G
\end{array}
\]
and $J|_{\mathscr{P}_G} \colon \mathscr{P}_G \to \mathscr{T}_G$ is therefore a branched double cover.

\begin{Thm} \label{t:irri}
Let~$G$ be a $2$-connected plane graph.  If the HOMFLY polynomial $\hos{D(G)}$ is reducible then
\begin{itemize}
\item
the number of edges of~$G$ is even;
\item
the number of vertices of~$G$ is even;
\item
the polynomial $T_G(x,1)$ is a square.
\end{itemize}
\end{Thm}

\begin{proof}
Let $\widetilde{T}_G(x_0,x_1,y_0,y_1) \in \mathbb{Z}[x_0,x_1,y_0,y_1]$ be the numerator of the Laurent polynomial $T_G \circ \Sigma^{-1}$.  By construction, the polynomial $\widetilde{T}_G$ is bihomogeneous of degree $h_1(G)$ in $y_0,y_1$ and of degree
\[
\deg_x T_G + \deg_y T_G =
\#V(G)-1+h_1(G) = \#E(G)
\]
in $x_0,x_1$.  The vanishing set of~$\widetilde{T}_G$ in $\ppu$ is the curve~$\mathscr{T}_G$.

Because the graph~$G$ is $2$-connected, the Tutte polynomial $T_G(x,y)$ is irreducible by \cite{MMN}*{Theorem~1}: even though the cited paper states irreducibility over~$\mathbb{Z}$, the authors mention, and their argument shows, that $T_G(x,y)$ is also irreducible in $\bC[x,y]$.  Since~$\Sigma$ is a birational map, the polynomial~$\widetilde{T}_G$ is also irreducible, and hence so is the curve~$\mathscr{T}_G$.  An equation of curve~$\mathscr{P}_G$ is $\widetilde{T}_G(x_0,x_1,y_0^2,y_1^2) = 0$.  If~$\mathscr{P}_G$ is reducible, then we are in a position to apply Lemma~\ref{l:riqua}.  We deduce that the number of edges of~$G$, the degree of~$\widetilde{T}_G$ with respect to $x_0,x_1$, is even.  Using Lemma~\ref{l:res}, we evaluate $\widetilde{T}_G(x,1,0,1)$ and we find
\[
\widetilde{T}_G(x,1,0,1) = T_G(1-x,1) .
\]
We obtain that the evaluation $T_G(1-x,1)$, or, equivalently, $T_G(x,1)$, is a square.  Finally, since the degree of $T_G(x,1)$ is $\#E(G)-h_1(G)-1$, and we already argued that $\#E(G)$ is even, we deduce that $h_1(G)-1$ is even.  Since~$G$ is connected, the identity $h_1(G) -1 = \#E(G)-\#V(G)$ holds.  As we already showed that $\#E(G)$ is even, we conclude that $\#V(G)$ is even and the proof is complete.
\end{proof}

\begin{Rem} \label{r:fap}
Let~$G$ be a finite graph.  Define a simplicial complex $F(G)$ on the edges of~$G$ by letting $\sigma \subset E(G)$ be a face of $F(G)$ if and only if $\sigma$ contains no cycle.  The evaluation $T_G(1-x,1)$ is the face polynomial of the simplicial complex $F(G)$.
\end{Rem}

\section{Further directions}

We found that every reducible \hop{} of a knot with at most~$12$ crossings is itself the product of irreducible \hop{}s of knots.  We would find it surprising if this was always the case.  Nevertheless, it would be interesting to study further the divisibility properties of \hop{}s of knots (or even of links).  At an experimental level, extensive tables of \hop{}s of knots and links are available, so gathering further evidence is easily within reach.  At a conceptual level, we would find it very interesting to predict factorizations of \hop{}s, without having to look them up in tables.

We could not find a 2-connected plane graph~$G$ with an even number of vertices and of edges and such that the evaluation $T_G(x,1)$ is a square, nor we could prove that they do not exist.  Our expectation is that such graphs do not exist.  If this were the case, then it would follow from Theorem~\ref{t:irri} that the \hop{}s of the oriented links associated to 2-connected plane graphs are all irreducible.  Using Remark~\ref{r:fap} we can reformulate one of the conditions on the graph saying that the face polynomial of a simplicial complex is a square.  We have never come across a similar condition.

\end{document}